\documentclass[11pt]{amsart}
\usepackage{latexsym,amssymb,amsmath,hyperref,empheq}
\usepackage{amsthm, dsfont}
\usepackage{tikz-cd}
\usepackage[all]{xy}
\usepackage[short,nodayofweek]{datetime}
\usepackage[active]{srcltx}

\leftmargin  \leftmargini
\def\ZZ{{\mathbb Z}}

\def\RR{{\mathbb R}}
\def\CC{{\mathbb C}}

\def\H{{\mathcal H}}

\DeclareMathOperator{\trace}{tr}

\DeclareMathOperator{\re}{Re}
\DeclareMathOperator{\Sp}{Sp}

\theoremstyle{plain}
\newtheorem{thm}{Theorem}[section]
\newtheorem{cor}[thm]{Corollary}
\newtheorem{lem}[thm]{Lemma}

\theoremstyle{definition}

\theoremstyle{remark}

\newtheoremstyle{Acknowledgements}
  {}
    {}
     {}
     {}
    {\bfseries}
    {}
     {.5em}
     {\thmname{#1}\thmnumber{ }\thmnote{ (#3)}}
\theoremstyle{Acknowledgements}

\date{\today, \currenttime} 

\begin{document}
\title[Sturm for scalar weight]{Sturm's operator for scalar weight in arbitrary genus}

\author{Kathrin Maurischat }
\address{\rm {\bf Kathrin Maurischat}, Mathematisches Institut,
   Heidelberg University, Im Neuenheimer Feld 205, 69120 Heidelberg, Germany }
\curraddr{}
\email{\sf maurischat@mathi.uni-heidelberg.de}
\thanks{This work was  supported by the European Social Fund.}

\begin{abstract}
In contrast to the wellknown cases of large weights, 
 Sturm's operator does not realize the holomorphic projection operator for lower weights.
 We prove its failure for arbitrary Siegel genus $m\geq 2$ and scalar weight $\kappa=m+1$.
 This generalizes a result for genus two in \cite{phantom}.
\end{abstract}

\maketitle
\setcounter{tocdepth}{2}
\tableofcontents

\section{Introduction}
Sturm's operator arises in questions concerning the description of holomorphic projection.
Let $\Sp_m(\RR)$ be the real symplectic group of genus $m$. Let $\Gamma$ be a subgroup of $\Sp_m(\ZZ)$ of finite index containing the group of translations $\Gamma_\infty$.
For an integer $\kappa$ we look at $L^2$-functions of weight $\kappa$, i.e.  transforming with the
 $\kappa$-th power of the determinant under right translations by the maximal compact subgroup $K$. The problem of holomorphic projection is  to describe the orthogonal projection operator
\begin{equation*}
 S\::\:L^2(\Gamma\backslash \Sp_m(\mathbb R))_\kappa \xrightarrow{\:\mathop{proj}\:} L^2(\Gamma\backslash \Sp_m(\mathbb R))_{\kappa, hol}\:,
\end{equation*}
where $L^2(\Gamma\backslash \Sp_m(\mathbb R))_{\kappa, hol}$ is that part of the spectrum annihilated by $\mathfrak p^-$, the negative eigenspace in the Cartan decomposition of the 
Lie algebra of $\Sp_m(\RR)$.
The only known tools for this are analytic ones. For a $C^\infty$-modular form $F$ on the Siegel halfspace $\H_m$, Sturm's operator produces coefficients $a(T)$ 
for each positive definite halfintegral $T$,
which turn out to be the Fourier coefficients of the holomorphic projection $S(F)$ if the weight $\kappa$ is large, i.e. $\kappa>2m$.
This is proved by an unfolding argument. The coefficients $a(T)$ are the values of the inner product of $F$ against a system of Poincar\'e series $P_T$, which spans the space of
holomorphic cuspforms $[\Gamma,\kappa]_0$ for $\Gamma$ of weight $\kappa$.
This method was invented by Sturm~\cite{sturm1}, \cite{sturm2} in the classical case $m=1$, and proved for arbitrary genus $m$ by Panchishkin~\cite{panchishkin}.
By analytic continuation (which is indeed holomorphic in these cases) of the Poincar\'e series  it is also seen to be true for classical $m=1$ and $\kappa=2$ (\cite{gross-zagier}), 
and for $m=2$ and $\kappa=4$ (\cite{holproj}), that is in the
case of low weight $\kappa=2m$.

But in \cite{phantom} it was shown that this method fails in case of genus $m=2$ and weight $\kappa=m+1=3$. First, the analytic continuations of the Poincar\'e series fail to be holomorphic,
but have  nonzero parts in the spectral component generated by a holomorphic cuspform of weight $\kappa-2$. Second, this nonzero share is not annihilated by Sturm's operator, 
but produces a so called phantom term.

There is a strong interplay between analysis and representation theory.
As spectral theory for small weights is  a difficult and mostly open problem, Sturm's operator can be used as an analytic indicator. 
If a nonholomorphic spectral component isn't mapped to zero by it, i.e. produces phantoms,
we expect that the analytic continuation of Poincar\'e series, which is unknown in general, will have a nonzero share in this spectral component, too.

In this paper we prove that Sturm's operator does not realize the holomorphic projection operator in the case of weight $\kappa=m+1$ for arbitrary genus $m\geq 2$.
More precisely, it shows the analogous behavior as in case $m=2$. If $h\in[\Gamma,\kappa-2]_0$ is a nonzero holomorphic cuspform of weight $\kappa-2$, then its image under
Maass' shift operator $\Delta_+^{[m]}$ is a nonholomorphic modular form of weight $\kappa$.   Under Sturm's operator of level $\kappa$ this function $\Delta_+^{[m]}(h)$ produces a nonzero 
phantom term. The arguments rely on properties of differential operators developped in~\cite{freitag}.

The question remains whether the images $\Delta_+^{[m]}([\Gamma, m-1]_0)$ are the only spectral components which contribute to the phantom terms for arbitrary genus.
Connected to this, it is also an interesting question whether Sturm's operator fails to realize the holomorphic projection for the remaining smaller weights $\kappa=m+2,\dots, 2m$.
It is a corollary (see Corollary~\ref{resultat_k>=m}) of our proof that for these weights the functions of $\Delta_+^{[m]}([\Gamma, \kappa-2]_0)$ do not produce phantoms.
There are obviously further candidates to test, the images $(\Delta_+^{[m]})^r (h)$ for holomorphic cuspforms $h\in[\Gamma, \kappa-2r]_0$ for positive weights $\kappa-2r>0$.
Although this test can in principle be done by the basic analytic tools we use here, the iterated application of $\Delta_+^{[m]}$ demands an enormous computational effort.
As we do not have a direct application for such results, we don't follow  them here. Further, these weren't the only functions to test. Phantoms could arise as well from spectral 
components given by vector valued holomorphic cuspforms or even from nonholomorphic spectral components containing the $K$-type $\kappa$ nontrivially.
So the result presented here which might look poorly at a first glance is the best available at the moment by reasonable costs.

\section{Notation and statement of results}
Let $G=\Sp_m(\RR)$ be the real symplectic group of genus $m$. Let $\Gamma\subset \Sp_m(\ZZ)$ be a discrete subgroup of $G$ containing the group of translations $\Gamma_\infty$ (unipotent upper
triangular matrices). Let $\mathcal H_m$ be the Siegel upper halfspace equipped with the ordinary action of $G$, 
$gZ=(aZ+b)(cZ+d)^{-1}$, and denote by $K\simeq U(m)$ the maximal compact subgroup of $G$ which stabilizes 
$i E_m\in\mathcal H_m$.
Let $F$ be a nonholomorphic modular cusp form on $\mathcal H_m$ for $\Gamma$ of scalar weight $\kappa$ with Fourier expansion
\begin{equation*}
 F(Z)\:=\:\sum_{T>0}A(T,Y)\exp(2\pi i\trace(TX))\:,
\end{equation*}
where the sum runs over all positive definite $(m,m)$-matrices $T$ with half integral entries.
Recall (\cite{phantom}) Sturm's operator for weight $\kappa$ in terms of Fourier coefficients
\begin{equation*}
 S\::\:A(T,Y)\:\mapsto\: a(T)\:,
\end{equation*}
\begin{equation*}
 a(T)\:=\: c(m,\kappa)^{-1}\int_{Y>0}A(T,Y)\exp(-2\pi\trace(TY))\det(TY)^{\kappa-\frac{m+1}{2}}dY_{inv}\:.
\end{equation*}
Here $dY_{inv}=\det(Y)^{-\frac{m+1}{2}}\prod_{j\leq k}dY_{jk}$ is the invariant measure on the symmetric space $\{ Y\in\mathop{Symm}\nolimits_m(\RR)\mid Y>0\}$. The constant
\begin{equation*}
 c(m,\kappa)\:=\: (4\pi)^{-m(\kappa-\frac{m+1}{2})}\Gamma_m(\kappa-\frac{m+1}{2})
\end{equation*}
is chosen such that on holomorphic forms, i.e. $A(T,Y)=a(T)\exp(-2\pi\trace(TY))$, Sturm's operator is the identity.
Define the formal sum
\begin{equation*}
 S_F(Z)\:=\:\sum_{T>0}a(T)\exp(2\pi i\trace(TZ))\:.
\end{equation*}
In many cases, i.e. for large $\kappa>2m$ (\cite{panchishkin}) or for $m\leq 2$ and $\kappa=2m$ (\cite{gross-zagier} for $m=1$, \cite{holproj} for $m=2$), 
Sturm's operator realizes the holomorphic projection operator. That is, $S_F \in[\Gamma,\kappa]_0$ defines a holomorphic cusp form of weight $\kappa$, and for all $f\in [\Gamma,\kappa]_0$
we have
\begin{equation*}
 \langle F,f\rangle\:=\:\langle S_F,f\rangle\:,
\end{equation*}
with respect to the scalar product  on $L^2_\kappa(\Gamma\backslash \mathcal H_m)$.
In~\cite{phantom} it was shown that this doesn't hold  anymore for $m=2$ and $\kappa=3$, because in this case  Sturm's operator applied to the non-holomorphic function 
\begin{equation*}
 \tilde h\::=\:\Delta_+^{[2]}( h)
\end{equation*}
has nonzero value.
Here $h\in[\Gamma,k]_0$ is a holomorphic cusp form of weight $k:=\kappa-2$, and $\Delta_+^{[m]}$ is the Maass differential operator
\begin{equation*}
 \Delta_+^{[m]}\::\: C^\infty(\mathcal H_m,V_\tau) \to C^\infty(\mathcal H_m, V_{\tau\otimes \det\nolimits^2})\:,
\end{equation*}
\begin{equation*}
 \Delta_+^{[m]} (h)(Z)\:=\: (2i)^m(\tau\otimes\det\nolimits^{\frac{1-m}{2}})(Y^{-1})\det(\partial_Z)\left((\tau\otimes\det\nolimits^{\frac{1-m}{2}})(Y)\cdot h(Z)\right),
\end{equation*}
which sends a form of weight $\tau$ to a form of weight $\tau\otimes\det\nolimits^2$.
We show that this holds in general for arbitrary genus $m$ and weight $\kappa=m+1$.
\begin{thm}\label{hauptresultat}
 Let $m\geq 2$ and let $\kappa=m+1$. Let $h\in[\Gamma,k]_0$ be a holomorphic cuspform of weight $k=\kappa-2$ with Fourier expansion
 \begin{equation*}
  h(Z)\:=\:\sum_{T>0}b(T)\exp(2\pi i\trace(TZ))\:.
 \end{equation*}
The image  under Sturm's operator for weight $\kappa$ of the function $\tilde h=\Delta_+^{[m]}( h)$ is given by the Fourier coefficients
\begin{equation*}
 a(T)\:=\:(-1)(4\pi)^m\det(T)\cdot b(T)\:. 
\end{equation*}
So 
the formal image is
\begin{equation*}
 S_{\tilde h}(Z)\:=\: (-1)^{m+1}(2i)^m\cdot \det(\partial_Z)(h(Z))\:.
\end{equation*}
\end{thm}
The automorphic representation generated by a holomorphic cusp form $h\in[\Gamma, k]_0$ does not contain holomorphic functions of weight $\kappa=k+2$. 
Thus the image $S_{\tilde h}$ is not the holomorphic projection of $\tilde h$, and
the following main result is a direct consequence of Theorem~\ref{hauptresultat}.
\begin{thm}\label{noholproj}
 Let $m$ be arbitrary and let $\kappa=m+1$. Then Sturm's operator does not realize the projection to  holomorphic components.
\end{thm}
In case $m=2$ it was shown in~\cite{phantom} that the components given by $\Delta_+^{[2]}[\Gamma,1]$ are indeed the only ones on which Sturm's operator for weight $3$ fails to be the 
holomorphic projection. This was done by giving a system of Poincar\'e series.
At the moment we cannot able to proof a correspondingly complete result for $m>2$.
\section{Proof of Theorem~\ref{hauptresultat}}\label{beweis_hauptresultat}
Theorem~\ref{hauptresultat} is proven by the same means as its analog~\cite[Prop.~5.2]{phantom} for $m=2$. But for arbitrary $m$ the book keeping is more involved.
Let 
\begin{equation*}
  h(Z)\:=\:\sum_{T>0}b(T)\exp(2\pi i\trace(TZ))\:
 \end{equation*}
 be a cuspform for $\Gamma$ of arbitrary scalar weight $k$. This Fourier expansion converges absolutely and locally uniformly. So to compute the image
 \begin{equation*}
 \tilde h(Z)\:=\:\Delta_+^{[m]} (h)(Z)\:=\: (2i)^m(\det\nolimits^{k+\frac{1-m}{2}})(Y^{-1})\det(\partial_Z)\left(\det\nolimits^{k+\frac{1-m}{2}}(Y)\cdot h(Z)\right)
\end{equation*}
it is enough to compute the Fourier coefficients $b(T,Y)$ given by
\begin{equation*}
\frac{b(T)(2i)^m}{\exp(2\pi i\trace(TZ))}(\det\nolimits^{k+\frac{1-m}{2}})(Y^{-1})\det(\partial_Z)\left(\det\nolimits^{k+\frac{1-m}{2}}(Y)\cdot \exp(2\pi i\trace(TZ)\right)
\end{equation*}
of $\tilde h(Z)=\sum_{T>0}b(T,Y)\exp(2\pi i\trace(TZ))$.
\begin{lem}\label{det-Ableitung-Produkt}
 The derivative
 \begin{equation*}
  \det(\partial_Z)\left(\det\nolimits^{j}(Y)\cdot \exp(2\pi i\trace(TZ)\right)
\end{equation*}
equals
\begin{equation*}
 (2i)^{-m}\det(Y)^{j-1}\exp(2\pi i\trace(TZ))\sum_{q=0}^m (-4\pi)^qC_{m-q}(j)\trace((Y^{\frac{1}{2}}TY^{\frac{1}{2}})^{[q]})\:.
\end{equation*} 
\end{lem}
Here $C_l(\alpha)=\alpha(\alpha+\frac{1}{2})\cdots(\alpha+\frac{l-1}{2})$, and for a  matrix $M$ the matrix $M^{[q]}$ is the matrix of its $q$-th exterior power, i.e. its entries are the 
$(q\times q)$-subminors of $M$.
Lemma~\ref{det-Ableitung-Produkt} is proven (see section \ref{beweis_von_det-Ableitung-Produkt}) by applying a sequence of arguments due to Freitag~\cite[III.6]{freitag}.

According to Lemma~\ref{det-Ableitung-Produkt} the Fourier coefficients of $\tilde h$ are 
\begin{equation*}
 b(T,Y)\:=\: b(T)\sum_{q=0}^m (-4\pi)^qC_{m-q}(k+\frac{1-m}{2})\det(Y)^{-1}\trace((Y^{\frac{1}{2}}TY^{\frac{1}{2}})^{[q]})\:.
\end{equation*}
Sturm's operator for weight $\kappa=k+2$ now is
\begin{equation*}
 S\::\: b(T,Y)\exp(-2\pi\trace(TY))\:\mapsto\: a(T)c(m,\kappa)^{-1}\:,
\end{equation*}
where $a(T)$ is given by the limit $s\to 0$ of the coefficients
\begin{equation*}
 a(T,s)\:=\:\int_{Y>0}b(T,Y)\exp(-4\pi\trace(TY))\det(TY)^{\kappa-\frac{m+1}{2}+s}~dY_{inv}\:.
\end{equation*}
\begin{lem}\label{Koeff_int}
Define
 \begin{equation*}
  I_q(s)\::=\:\int_{Y>0} \trace((Y^{\frac{1}{2}}TY^{\frac{1}{2}})^{[q]})\det(TY)^s\exp(-4\pi\trace(TY))~dY_{inv}\:.
 \end{equation*}
For complex $s$ with $\re s>>0$ such that the integral exists we have
\begin{equation*}
 I_q(s)\:=\:(-4\pi)^{-q}(4\pi)^{-ms}\binom{m}{q}C_q(-s)\Gamma_m(s)\:. 
\end{equation*}
\end{lem}
\begin{proof}[Proof of Lemma~\ref{Koeff_int}]
 By the invariance of the measure $dY_{inv}$ we have
 \begin{eqnarray*}
  I_q(s)&=&
  \int_{Y>0} \trace(Y^{[q]})\det(Y)^s\exp(-4\pi\trace(Y))~dY_{inv}\\
  &=&
  (4\pi)^{-(q+ms)} \int_{Y>0} \trace(Y^{[q]})\det(Y)^s\exp(-\trace(Y))~dY_{inv}\:.
 \end{eqnarray*}
The last integral is seen to equal $(-1)^q\binom{m}{q}C_q(-s)\Gamma_m(s)$ by Lemma~\ref{q-trace-integral}.
\end{proof}
According to Lemma~\ref{Koeff_int} we write
\begin{equation*}
 a(T,s)\:=\: b(T)\det(T)\sum_{q=0}^m C_{m-q}(k+\frac{1-m}{2})(-4\pi)^qI_q(\kappa-\frac{m+1}{2}+s-1)\:,
\end{equation*}
respectively, (recalling $\kappa=k+2$) this equals
\begin{equation*}
b(T)\det(T)\frac{\Gamma_m(k+\frac{1-m}{2}+s)}{(4\pi)^{m(k+\frac{1-m}{2}+s)}}\sum_{q=0}^m \binom{m}{q}C_{m-q}(k+\frac{1-m}{2})C_q(-(k+1-\frac{m+1}{2}+s))\:.
\end{equation*}
\begin{lem}\label{big_sum}
For an interger $m>0$ define the polynomial
\begin{equation*}
 P_m(s,z)\:=\:\sum_{q=0}^m\binom {m}{q}(-1)^q\prod_{j=0}^{m-q-1}(z+\frac{j}{2})\prod_{j=0}^{q-1}(z+s-\frac{j}{2})\:.
\end{equation*}
Then for all $s,z\in\CC$ it holds
\begin{equation*}
 P_m(s,z)\:=\: (-1)^ms(s-\frac{1}{2})\cdots(s-\frac{m-1}{2})\:.
\end{equation*} 
\end{lem}
\begin{proof}[Proof of Lemma~\ref{big_sum}]
We have $P_1(s,z)=-s$.
  Applying the identity $\binom{m+1}{q}=\binom{m}{q}+\binom{m}{q-1}$ to
 \begin{equation*}
  P_{m+1}(s,z)\:=\:\sum_{q=0}^{m+1}\binom {m+1}{q}(-1)^q\prod_{j=0}^{m-q}(z+\frac{j}{2})\prod_{j=0}^{q-1}(z+s-\frac{j}{2})
 \end{equation*}
we get
\begin{eqnarray*}
 P_{m+1}(s,z) &=&
 \sum_{q=0}^m\binom {m}{q}(-1)^q\prod_{j=0}^{m-q-1}(z+\frac{1}{2}+\frac{j}{2})\prod_{j=0}^{q-1}(z+s-\frac{j}{2})\\
 && -\sum_{q=0}^{m-1}\binom {m}{q}(-1)^q\prod_{j=0}^{m-q-1}(z+\frac{j}{2})\prod_{j=0}^{q}(z+s-\frac{j}{2})\\
 && +(-1)^{m+1}\prod_{j=0}^m(z+s-\frac{j}{2})\:. 
\end{eqnarray*}
This is
\begin{equation*}
 P_{m+1}(s,z)\:=\:z\cdot P_m(s-\frac{1}{2},z+\frac{1}{2})-(z+s) P_m(s-\frac{1}{2},z)\:.
\end{equation*}
Assuming by induction  the lemma to be true for  $m$ and all $s,z\in\CC$, we get
\begin{equation*}
 P_{m+1}(s,z)\:=\:(-1)^m(z-(z+s))(s-\frac{1}{2})\cdots(s-\frac{m}{2})\:=\:(-1)^{m+1}s(s-\frac{1}{2})\cdots(s-\frac{m}{2})\:.
\end{equation*} 
\end{proof}
Recalling $C_l(\alpha)=\alpha(\alpha+\frac{1}{2})\cdots(\alpha+\frac{l-1}{2})$ we have by Lemma~\ref{big_sum}
\begin{equation*}
 a(T,s)\:=\: b(T)\det(T)\frac{\Gamma_m(k+\frac{1-m}{2}+s)}{(4\pi)^{m(k+\frac{1-m}{2}+s)}}P_m(s,k+\frac{1-m}{2})\:.
\end{equation*}
So the image of Sturm's operator is given by
\begin{equation*}
 \lim_{s\to 0} a(T,s)c(m,\kappa)^{-1}
 \end{equation*}
which equals
 \begin{equation}\label{bild_gleichung}
\frac{(-1)^mb(T)\det(T)}{c(m,\kappa)}\cdot
\lim_{s\to 0} \frac{\Gamma_m(k+\frac{1-m}{2}+s)}{(4\pi)^{m(k+\frac{1-m}{2}+s)}}\cdot s(s-\frac{1}{2})\cdots(s-\frac{m-1}{2})\:. 
\end{equation}
Recall the product formula for $\Gamma_m$,
\begin{equation*}
 \Gamma_m(s)\:=\:\pi^{\frac{m(m-1)}{4}} \prod_{\nu=0}^{m-1}\Gamma(s-\frac{\nu}{2})\:.
\end{equation*}
For the proof of Theorem~\ref{hauptresultat} we now specialize to the case $k=m-1$, thus $\kappa=m+1$. In this case the above limit~(\ref{bild_gleichung}) is
\begin{equation*}
 \frac{(-1)^mb(T)\det(T)\pi^{\frac{m(m-1)}{4}}}{c(m,k+2)(4\pi)^{\frac{m(m-1)}{2}}}\cdot\prod_{\nu=0}^{m-2}\left(\Gamma(\frac{m-1-\nu}{2})(-\frac{\nu+1}{2})\right)
 \cdot\lim_{s\to 0}s\cdot\Gamma(s)\:.
\end{equation*}
Recalling $c(m,\kappa)=(4\pi)^{-m(k+\frac{1-m}{2}+1)}\Gamma_m(k+\frac{1-m}{2}+1)$ this simplyfies to
\begin{equation*}
 \lim_{s\to 0} a(T,s)c(m,\kappa)^{-1}\:=\: (-1)(4\pi)^mb(T)\det(T)\:.
 \end{equation*}
This proves Theorem~\ref{hauptresultat} for Fourier coefficients. The formula for the image function  follows by 
\begin{equation*}
 \det(\partial_Z)\left(\sum_{T>0}b(T)\exp(2\pi i\trace(TZ))\right)\:=\:  (2\pi i)^m\sum_{T>0}b(T)\det(T)\exp(2\pi i\trace(TZ))\:.
\end{equation*}

For scalar weights $k\geq m$, we get the following corollary of the proof of Theorem~\ref{hauptresultat}.
\begin{cor}\label{resultat_k>=m}
 For all holomorphic cusp forms $h\in[\Gamma,k]_0$ of weight $k\geq m$ the image of $\Delta_+^{[m]} (h)$ under Sturm's operator vanishes.
\end{cor}
\begin{proof}[Proof of Corollary~\ref{resultat_k>=m}]
The image under Sturm's operator in terms of Fourier coefficients is given by (\ref{bild_gleichung}).
But the Gamma function $\Gamma_m(z)$ is holomorphic at $z=k+\frac{1-m}{2}$ for $k\geq m$.
So evaluating the limit~(\ref{bild_gleichung}) in this case yields zero, i.e. the image of Sturm's operator vanishes.
\end{proof}
The corollary has the following consequence for Sturm's operator of weight $\kappa\geq m+2$. There are no phantom terms arising from the spectral components corresponding to $[\Gamma,\kappa-2]_0$.

\subsection{Proof of Lemma~\ref{det-Ableitung-Produkt}}\label{beweis_von_det-Ableitung-Produkt}
Recall that for a $(m,m)$-matrix $M$ its $q$-th exterior power is given by $M^{[q]}$, the $(\binom{m}{q},\binom{m}{q})$-matrix having entries given by $(q\times q)$-minors of $M$ indexed by
$M^{[q]}_{a,b}$ for subsets $a,b\subset\{1,\dots, m\}$ of cardinality $\lvert a\rvert=\lvert b\rvert=q$.
Especially, for $q=m$ we get $M^{[m]}=\det (M)$.

We follow the notation of~\cite[III.6]{freitag}.
For endomorphisms $A$ and $B$ of the exterior powers of a complex vetorspace, $A:\bigwedge^p V\to \bigwedge^p V$ and $B:\bigwedge^q V\to \bigwedge^q V$, let 
\begin{equation*}
 A \sqcap B\::\: \bigwedge\nolimits^{p+q} V\to \bigwedge\nolimits^{p+q}V
\end{equation*}
the induced endomorphism on $\bigwedge^{p+q}$.
By \cite[p. 207f]{freitag} the matrix entries of $A \sqcap B$ with respect to those of $A$ and $B$ are 
\begin{equation}\label{verbinde_A_mit_B}
 (A \sqcap B)_{a,b}\:=\:\binom{p+q}{p}^{-1}\sum_{a',b'}\epsilon(a',a'')\epsilon(b',b'')A_{a',b'}B_{a'',b''}\:,
\end{equation}
where $a',a''$ are subsets of $a$ such that $a'\cup a''=a$ and 
$b',b''$ are subsets of $b$ such that $b'\cup b''=b$ with $\lvert a'\rvert=\lvert b'\rvert=p$ and $\lvert a''\rvert=\lvert b''\rvert=q$.
The sign $\epsilon(a',a'')=\pm 1$ is that of the permutation of $p+q$ elements which for natural ordered sets $a'$ and $a''$ produces the natural order on $(a',a'')$.
For a real symmetric  matrix variable $T=(t_{\mu\nu})_{\mu\nu}$ the derivative $\partial_T$ has entries ${\partial_T}_{\mu\nu}=\frac{1}{2}(1+\delta_{\mu\nu})\frac{\partial}{\partial_{t_{\mu\nu}}}$,
and for complex symmetric $Z=X+iY$ we have $\partial_Z=\frac{1}{2}(\partial_X-i\partial_Y)$.
\begin{lem}\label{freitag_kollektion}\cite[pp. 210, 211, 213]{freitag}
 Let $Y$ be a  symmetric $(m,m)$-matrix variable. For all symmetric $(m,m)$-matrices $T$ it holds
 \begin{equation}\label{exp_ableitung}
  \partial_Y^{[q]}\left(\exp(\trace(TY))\right)\:=\: T^{[q]}e^{\trace(TY)}\:.
 \end{equation}
Further,
\begin{equation}\label{det_ableitung}
 \partial_Y^{[q]}(\det(Y)^\alpha)\:=\:C_q(\alpha)\det(Y)^\alpha (Y)^{-[q]}\:,
\end{equation}
where $C_q(\alpha)=\alpha(\alpha+\frac{1}{2})\cdots(\alpha+\frac{q-1}{2})$.
For $C^\infty$-functions $f$ and $g$ in $Y$ it holds
\begin{equation}\label{prod_ableitung}
 \partial_Y^{[h]}(f\cdot g)\:=\:\sum_{p+q=h}\binom{h}{p}(\partial_Y^{[p]}f)\sqcap(\partial_Y^{[q]}g)\:.
\end{equation}
\end{lem}
We apply Lemma~\ref{freitag_kollektion} for $f(Z)=\det(Y)^j$ and $g(Z)=\exp(2\pi i\trace(TZ))$.
\begin{eqnarray*}
 &&\det(\partial_Z)\left(\det(Y)^j\cdot\exp(2\pi i\trace(TZ))\right)\\
 && =\: \sum_{p+q=m}\binom{m}{p}(\partial_Z^{[p]}\det(Y)^j)\sqcap(\partial_Z^{[q]}\exp(2\pi i\trace(TZ)))\\
 && =\: \sum_{p+q=m}\binom{m}{p}C_p(j)\frac{(2\pi i)^q}{(2i)^p}\det(Y)^j\exp(2\pi i\trace(TZ))\left((Y^{-1})^{[p]}\sqcap T^{[q]}\right) 
\end{eqnarray*}
\begin{lem}\label{verbindung_Yinvers_mit_T}\cite[pp. 209, 215]{freitag}
 Let $0\leq p+q=h\leq m$ and let $Y$ and $T$ be symmetric positive definite $(m,m)$-matrices. Let $Y^{\frac{1}{2}}$ be the symmetric positive definite 
 matrix such that $Y^{\frac{1}{2}}Y^{\frac{1}{2}}=Y$.  Then it holds
 \begin{equation*}
  \left(Y^{-[p]}\sqcap T^{[q]}\right)Y^{[h]}\:=\:(Y^{-\frac{1}{2}})^{[h]}\left(E_m^{[p]}\sqcap T[Y^{\frac{1}{2}}]^{[q]}\right)(Y^{\frac{1}{2}})^{[h]}\:.
 \end{equation*}
\end{lem}
Combining (\ref{verbinde_A_mit_B}) and Lemma~\ref{verbindung_Yinvers_mit_T} for $h=m$ we get
\begin{eqnarray*}
 \left(Y^{-[p]}\sqcap T^{[q]}\right)
 &=&
 \det(Y)^{-1}\left( E_m^{[p]}\sqcap (Y^{\frac{1}{2}}TY^{\frac{1}{2}})^{[q]}\right)\\
 &=&
 \frac{1}{\binom{m}{p}\det(Y)}\sum_{a\subset\{1,\dots,m\},\lvert a\rvert=p}\epsilon(a,a^c)^2(Y^{\frac{1}{2}}TY^{\frac{1}{2}})^{[q]}_{a^c,a^c}\\
 &=&
 \frac{1}{\binom{m}{p}\det(Y)}\trace((Y^{\frac{1}{2}}TY^{\frac{1}{2}})^{[q]})\:.
\end{eqnarray*}
It follows that
\begin{eqnarray*}
 &&\det(\partial_Z)\left(\det(Y)^j\cdot\exp(2\pi i\trace(TZ))\right)\\
 && =\: (2i)^{-m}\det(Y)^{j-1}\exp(2\pi i\trace(TZ))\sum_{q=0}^m C_{m-q}(j)(-4\pi)^q\trace((Y^{\frac{1}{2}}TY^{\frac{1}{2}})^{[q]})\:,
\end{eqnarray*}
which is Lemma~\ref{det-Ableitung-Produkt}.
\bigskip

We include another lemma used in the proof of Theorem~\ref{hauptresultat}.
\begin{lem}\label{q-trace-integral}
 It holds
 \begin{equation*}
  \int_{Y>0}Y^{[q]}\exp(-\trace(Y))\det(Y)^s~dY_{inv}\:=\:(-1)^qC_q(-s)\Gamma_m(s)E_{\binom{m}{q}}\:.
 \end{equation*}
\end{lem}
\begin{proof}[Proof of Lemma~\ref{q-trace-integral}]
 Differentiating both sides of the wellknown identity
 \begin{equation*}
  \int_{Y>0}\exp(-\trace(TY))\det(Y)^s~dY_{inv}\:=\:\det(T)^{-s}\Gamma_m(s)
 \end{equation*}
by $\partial_T^{[q]}$ we get by using Lemma~\ref{freitag_kollektion} (\ref{exp_ableitung}) and (\ref{det_ableitung})
the identity
\begin{equation*}
 \int_{Y>0}Y^{[q]}\exp(-\trace(TY))\det(Y)^s~dY_{inv}\:=\: (-1)^qC_q(-s)\Gamma_m(s)\det(T)^{-s}T^{-[q]}\:.
\end{equation*}
Evaluating this at $T=E_m$ yields the lemma.
\end{proof}


\end{document}